\documentclass[11pt]{article}

\usepackage[pagebackref,colorlinks=true,pdfpagemode=none,urlcolor=blue,
linkcolor=blue,citecolor=blue]{hyperref}

\usepackage{amsmath,amsfonts,amssymb,amsthm}
\usepackage{mathrsfs}
\usepackage{color}

\newtheorem{theorem}{Theorem}[section]
\newtheorem{lemma}[theorem]{Lemma}

\newtheorem{assumption}[theorem]{Assumption}
\newtheorem{proposition}[theorem]{Proposition}

\theoremstyle{remark}
\newtheorem{remark}[theorem]{Remark}

\renewenvironment{proof}[1][Proof]{ {\itshape \noindent {#1.}} }{$\Box$
\medskip}

\numberwithin{equation}{section}
\newcommand{\R}{\mathbb{R}}

\newcommand{\Pb}{\mathbb{P}}
\newcommand{\E}{\mathbb{E}}
\newcommand{\F}{\mathcal{F}}

\newcommand{\G}{\mathcal{G}}

\newcommand{\C}{\mathcal{C}}

\newcommand{\V}{\mathbb{V}}

\newcommand{\RR}{\mathcal{R}}
\newcommand{\eps}{\varepsilon}
\def\les{\lesssim}

\begin{document}

\title{Homogenization of Parabolic Equations with\\ Large Time-dependent Random Potential}

\author{Yu Gu\thanks{Department of Applied Physics \& Applied
Mathematics, Columbia University, New York, NY 10027 (yg2254@columbia.edu; gb2030@columbia.edu)}  \and Guillaume Bal\footnotemark[1]}

\maketitle

\begin{abstract}
This paper concerns the homogenization problem of a parabolic equation with large, time-dependent, random potentials in high dimensions $d\geq 3$. Depending on the competition between temporal and spatial mixing of the randomness, the homogenization procedure turns to be different. We characterize the difference by proving the corresponding weak convergence of Brownian motion in random scenery. When the potential depends on the spatial variable macroscopically, we prove a convergence to SPDE.
\end{abstract}

\section{Introduction}

Small scales abound in equations of physical importance, where homogenization has become popular to analyze the asymptotics and reduce the complexity. For parabolic equation with random coefficients, when the solution can be expressed as the average with respect to certain diffusion process, homogenization may be recast as a problem of weak convergence of random motion in random environment, see a good introduction in \cite{komorowski2012fluctuations}.

The equation we consider in this paper is of the form $\partial_t u_\eps(t,x)=\Delta u_\eps(t,x)+iV_\eps(t,x)u_\eps(t,x)$. The size of $V_\eps$ is chosen large enough to produce some non-trivial effects on the asymptotic limit. The imaginary unit brings stability, and saves the effort of controlling the unbounded exponential function. Similar types of equations, including the ones with time-independent or real potentials, are analyzed in \cite{bal2009convergence, bal2010homogenization, B-AMRX-11, pardoux2006homogenization, pardoux2012homogenization,hairer2013random, gu2013weak}, using analytic and probabilistic methods, see a review in \cite{gu2013george}. The method we use here is probabilistic, i.e., a Feynman-Kac representation and weak convergence approach.

By a Feynman-Kac formula, a key object to analyze is the so-called Brownian motion in random scenery, i.e., the random process of the form $\int_0^t V_\eps(s,B_s)ds$, which corresponds to Kesten's model of random walk in random scenery in the discrete setting \cite{kesten1979limit}. If $V_\eps(t,x)\sim V(t/\eps^\alpha,x/\eps)$, by a simple change of variables, we observe a threshold of $\alpha=2$, which separates the effects of the random mixings generated by temporal and spatial variables. In other words, depending on whether $\alpha>2$ or $\alpha<2$, the averaging of $\int_0^t V_\eps(s,B_s)ds$ is induced by the temporal or spatial mixing of $V$, respectively. As a result, the ways we prove weak convergence of $\int_0^tV_\eps(s,B_s)ds$ are forced to be different correspondingly. When $\alpha>2$, it is a standard proof of central limit theorem for functions of mixing processes \cite[Chapter 4]{billingsley2009convergence} after freezing the Brownian motion. When $\alpha\leq 2$, the spatial mixing dominates. We make use of the Brownian motion by the Kipnis-Varadhan's method \cite{kipnis1986central}, i.e., constructing the corrector function and applying martingale decomposition. When $\alpha=2$, an ergodicity suffices to pass to the limit; when $\alpha<2$, a quantitative martingale central limit theorem \cite{mourrat2012kantorovich} is applied. As $\alpha\to \infty$, the spatial mixing tends to zero, so heuristically the Brownian motion remains in the weak convergence limit of $\int_0^t V_\eps(s,B_s)ds$, leading to a stochastic equation.

The rest of the paper is organized as follows. In Section \ref{sec:mainTH}, we introduce the problem setup and present the main results. Section \ref{sec:alphaG2}, \ref{sec:alphaS2} and \ref{sec:alphaInfty} are devoted to the cases $\alpha\in (2,\infty)$, $\alpha\in [0,2]$ and $\alpha=\infty$, respectively. Technical lemmas are left in the Appendix.

Here are notations used thoughout the paper. Since we have two independent random sources, i.e., the random potential from the equation and the Brownian motion induced by the Feynman-Kac formula, we are in the product probability space. $\E$ is used to denote the expectation with respect to the random environment, and $\E_B$ the expectation with respect to the Brownian motion. Joint expectation is denoted by $\E\E_B$. We use $a\wedge b=\min(a,b)$ and $a\vee b=\max(a,b)$. $N(\mu,\sigma^2)$ denotes the normal distribution with mean $\mu$ and variance $\sigma^2$, and $q_t(x)$ is the density of $N(0,t)$. We use $I_d$ to denote the $d\times d$ identity matrix. $a\les b$ stands for $a\leq Cb$ for some constant $C$ independent of $\eps$.

\section{Problem setup and main results}
\label{sec:mainTH}

We are interested in equations of the form
\begin{equation}
\partial_t u_\eps(t,x)=\frac12\Delta u_\eps(t,x)+i\frac{1}{\eps^\delta}V(\frac{t}{\eps^\alpha},\frac{x}{\eps})u_\eps(t,x),
\label{eq:mainEq}
\end{equation}
where $\alpha,\delta>0$, the dimension $d\geq 3$, and $V(t,x)$ is a mean-zero, time-dependent, stationary random potential. The initial condition $u_\eps(0,x)=f(x)\in \C_b(\R^d)$. Without loss of generality, we assume $f$ taking values in $\R$. For fixed $\alpha$, $\delta$ is chosen so that the large, highly oscillatory, random potential generates non-trivial effects on $u_\eps$ as $\eps\to 0$. It turns out that for different values of $\alpha$, the asymptotic effects may be different, and the ways we prove convergence are different as well.

Before presenting the main result, we make some assumptions on $V(t,x)$.

Let $(\Omega,\mathcal{F},\Pb)$ be a \emph{random medium} associated with a group of measure-preserving, ergodic transformations $\{\tau_{(t,x)},t\in \R, x\in\R^d\}$. Let $\V\in L^2(\Omega)$ with $\int_\Omega \V(\omega)\Pb(d\omega)=0$. Define $V(t,x,\omega)=\V(\tau_{(t,x)}\omega)$. The inner product and norm of $L^2(\Omega)$ are denoted by $\langle.,.\rangle$ and $\|.\|$, respectively.

By defining $T_{(t,x)}$ on $L^2(\Omega)$ as $T_{(t,x)}f(\omega)=f(\tau_{(t,x)}\omega)$ and assuming it is strongly continuous in $L^2(\Omega)$, we obtain the spectral resolution
\begin{equation}
T_{(t,x)}=\int_{\R^d}e^{i\xi_0 t+i\xi\cdot x}U(d\xi_0,d\xi),
\end{equation}
where $\xi=(\xi_1,\ldots,\xi_d)$ and $U(d\xi_0,d\xi)$ is the associated projection valued measure. Let $\{D_k,k=0,\ldots,d\}$ be the $L^2(\Omega)$ generator of $T_{(t,x)}$.

Let $\hat{R}(\xi_0,\xi)$ be the power spectrum associated with $\V$, i.e., $\hat{R}(\xi_0,\xi)d\xi_0d\xi=(2\pi)^{d+1}\langle U(d\xi_0, d\xi)\V,\V\rangle$, we obtain
\begin{equation}
R(t,x)=\E\{V(t,x)V(0,0)\}=\langle T_{(t,x)}\V,\V\rangle=\frac{1}{(2\pi)^{d+1}}\int_{\R^{d+1}}e^{i\xi_0 t+i\xi\cdot x}\hat{R}(\xi_0,\xi)d\xi_0d\xi.
\end{equation}

For any set $S\subseteq \R^{d+1}$, define the $\sigma-$algebras: $\F_S=\sigma(V(s,x):(s,x)\in S)$, and we assume a finite range of dependence of $V$.

\begin{assumption}
$V$ is uniformly bounded, and there exists $M>0$ such that $\F_{S_1}$ and $\F_{S_2}$ are independent for any two sets satisfying $dist(S_1,S_2)\geq M$, where $dist(S_1,S_2)$ is the Euclidean distance between $S_1,S_2$.
\label{ass:mixing}
\end{assumption}

Clearly the finite range of dependence implies that the following $\varphi-$mixing coefficient
\begin{equation}
\varphi(r):=\sup_{A\in \F_{S_1}, B\in \F_{S_2}, \Pb(B)>0, dist(S_1,S_2)\geq r}|\Pb(A|B)-\Pb(A)|
\end{equation}
is compactly supported on $[0,\infty)$.  Since $\varphi$ is bounded, we have $\varphi(r)\les 1\wedge r^{-n}$ for any $n>0$. We also have the $\rho-$mixing property by \cite[Page 170, Lemma 1]{billingsley2009convergence}:
\begin{equation}
|\E\{XY\}-\E\{X\}\E\{Y\}|\leq 2\varphi^\frac12(r)\left(\E\{X^2\}\E\{Y^2\}\right)^\frac12,
\end{equation}
if  $X$ is $\F_{S_1}-$measurable and $Y$ is $\F_{S_2}-$measurable with $dist(S_1,S_2)\geq r$.


\begin{remark}
Retracing the proof, one can see that instead of the finite range dependence, we only need $\varphi-$mixing in $t$ and $\rho-$mixing in $(t,x)$, with the corresponding mixing coefficients decaying sufficiently fast.
\end{remark}

In the following, we denote
\begin{eqnarray}
\sigma(V(s,x):s\leq t, x\in \R^d)=\F_t,\\
\sigma(V(s,x):s\geq t, x\in \R^d)=\F^t.
\end{eqnarray}


The following is the main result.
\begin{theorem}[$\alpha\in [0,\infty)$: \emph{homogenization}]
Under Assumption \ref{ass:mixing}, let
\begin{eqnarray}
\partial_t u_\eps(t,x)&=&\frac12\Delta u_\eps(t,x)+i\frac{1}{\eps^{\frac{\alpha}{2}\vee1}}V(\frac{t}{\eps^\alpha},\frac{x}{\eps})u_\eps(t,x),\\
\partial_t u_0(t,x)&=&\frac12\Delta u_0(t,x)-\rho(\alpha) u_0(t,x),
\end{eqnarray}
with initial condition $u_\eps(0,x)=u_0(0,x)=f(x)$ and
\begin{equation}
\rho(\alpha)=\left\{
\begin{array}{ll}
\int_0^\infty R(t,0)dt & \alpha\in (2,\infty),\\
\int_0^\infty \E_B\{R(t,B_t)\}dt & \alpha=2,\\
\int_0^\infty \E_B\{R(0,B_t)\}dt & \alpha\in [0,2).
\end{array} \right.
\end{equation}
 Then $u_\eps(t,x)\to u_0(t,x)$ in probability as $\eps \to 0$.
 \label{thm:homo}
\end{theorem}

When $\alpha>2$, by a change of parameter $\eps^\alpha\mapsto \eps$, we have $\frac{1}{\eps^\frac{\alpha}{2}}V(\frac{t}{\eps^\alpha},\frac{x}{\eps})\mapsto \frac{1}{\sqrt{\eps}}V(\frac{t}{\eps}, \frac{x}{\eps^{\frac{1}{\alpha}}})$, so $\alpha=\infty$ corresponds to the case when $V$ has no micro-structure in the spatial variable, and the potential is of the form $\frac{1}{\sqrt{\eps}}V(\frac{t}{\eps},x)$. We obtain a transition from homogenization to convergence to SPDE in the following theorem.

\begin{theorem}[$\alpha=\infty$: \emph{convergence to SPDE}]
Under Assumption \ref{ass:mixing}, let
\begin{eqnarray}
\partial_t u_\eps(t,x)&=&\frac12\Delta u_\eps(t,x)+i\frac{1}{\sqrt{\eps}}V(\frac{t}{\eps},x)u_\eps(t,x),\\
\partial_t u_0(t,x)&=&\frac12\Delta u_0(t,x)+i\dot{W}(t,x)\circ u_0(t,x),
\end{eqnarray}
with initial condition $u_\eps(0,x)=u_0(0,x)=f(x)$ and Gaussian noise $\dot{W}(t,x)$ of covariance structure $\E\{\dot{W}(t,x)\dot{W}(s,y)\}=\delta(t-s)\int_\R R(t,x-y)dt$. Then $u_\eps(t,x)\Rightarrow u_0(t,x)$ in distribution as $\eps\to0$.
\label{thm:spde}
\end{theorem}

\begin{remark}
The product $\circ$ in the limiting SPDE is in the Stratonovich's sense. 
\end{remark}

The solution to \eqref{eq:mainEq} is written by Feynman-Kac formula as
\begin{equation}
u_\eps(t,x)=\E_B\{f(x+B_t)\exp(i\frac{1}{\eps^\delta}\int_0^tV(\frac{t-s}{\eps^\alpha},\frac{x+B_s}{\eps})ds)\}.
\end{equation}
Since Theorem \ref{thm:homo} and \ref{thm:spde} are both results for fixed $(t,x)$, by stationarity of $V$, $u_\eps(t,x)$ has the same distribution as
\begin{equation}
\begin{aligned}
\tilde{u}_\eps(t,x)=&\E_B\{f(x+B_t)\exp(i\frac{1}{\eps^\delta}\int_0^tV(\frac{-s}{\eps^\alpha},\frac{x+B_s}{\eps})ds)\}\\
=&\E_B\{f(x+B_t)\exp(i\frac{1}{\eps^\delta}\int_0^t\tilde{V}(\frac{s}{\eps^\alpha},\frac{x+B_s}{\eps})ds)\},
\end{aligned}
\end{equation}
where $\tilde{V}(s,x):=V(-s,x)$. Since $\tilde{V}$ and $V$ has the same covariance function and mixing property we need, from now on we will write our solution to \eqref{eq:mainEq} for simplicity as
\begin{equation}
u_\eps(t,x)=\E_B\{f(x+B_t)\exp(i\frac{1}{\eps^\delta}\int_0^tV(\frac{s}{\eps^\alpha},\frac{x+B_s}{\eps})ds)\}.
\end{equation}

\subsection{Remarks on low dimensional cases}

The case $d=1$ with a real potential is addressed in \cite{pardoux2012homogenization, hairer2013random} using probabilistic and analytic approaches respectively. Their result shows that $\alpha\in(0,\infty)$ leads to homogenization while $\alpha=0,\infty$ leads to SPDE. When $\alpha>2$, our proof is similar to that in \cite{pardoux2012homogenization}. When $\alpha\in [0,2]$, we follow the approach in \cite{gu2013weak}, which relies on the fact $d\geq 3$.

For $d=2$, while the case of $\alpha>2$ can be analyzed in the same way, our approach for $\alpha\in [0,2]$ does not necessarily work. For the time-independent potential, weak convergence results are obtained in \cite{remillard1991limit,gu2013invariance}.

\section{$\alpha\in (2,\infty)$: temporal mixing and homogenization}
\label{sec:alphaG2}

When $\alpha>2$, $\delta=\frac{\alpha}{2}$. By Feynman-Kac formula, the solution is written as
\begin{equation}
u_\eps(t,x)=\E_B\{f(x+B_t)\exp(i\eps^{-\alpha/2}\int_0^t V(s/\eps^\alpha,(x+B_s)/\eps)ds)\}.
\end{equation}
By the scaling property of Brownian motion, stationarity of $V$, and a change of parameter $\eps^{\alpha}\mapsto \eps^2$, we have
\begin{equation}
u_\eps(t,x)=\E_B\{f(x+\eps B_{t/\eps^2})\exp(i\eps\int_0^{t/\eps^2} V(s,\eps^\beta B_s)ds)\},
\end{equation}
with $\beta=1-\frac{2}{\alpha}\in (0,1)$. Since $\beta>0$, the spatial mixing from $V$ for the process $\eps\int_0^{t/\eps^2} V(s,\eps^\beta B_s)ds$ is small.

Let $\sigma=\sqrt{2\int_0^\infty R(t,0)dt}$. The goal in this section is to prove that $u_\eps(t,x)\to u_0(t,x)$ in probability with $u_0$ satisfying
\begin{equation}
\partial_t u_0(t,x)=\frac12\Delta u_0(t,x)-\frac12\sigma^2 u_0(t,x).
\end{equation}

The result comes from the following two propositions.

\begin{proposition}
$$(\eps B_{t/\eps^2}, \eps\int_0^{t/\eps^2} V(s,\eps^\beta B_s)ds)\Rightarrow (N^1_t,\sigma N^2_t),$$ where $N^1_t\sim N(0,tI_d)$, independent from $N^2_t\sim N(0,t)$. The weak convergence $\Rightarrow$ is in the annealed sense.
\label{prop:conExpg2}
\end{proposition}
\begin{proposition}
For independent Brownian motions $B_t^1,B_t^2$, $$(\eps B_{t/\eps^2}^1, \eps B_{t/\eps^2}^2, \eps\int_0^{t/\eps^2} V(s,\eps^\beta B_s^1)ds-\eps\int_0^{t/\eps^2} V(s,\eps^\beta B_s^2)ds)\Rightarrow (N^1_t,N^2_t, \sqrt{2}\sigma N^3_t),$$ where $N^1_t, N^2_t\sim N(0,tI_d)$, $N^3_t\sim N(0,t)$, and they are independent. The weak convergence $\Rightarrow$ is in the annealed sense.
\label{prop:conVarg2}
\end{proposition}

By Proposition \ref{prop:conExpg2}, we have $\E\{u_\eps(t,x)\}\to u_0(t,x)$; by Proposition \ref{prop:conVarg2}, we have $\E\{|u_\eps(t,x)|^2\}\to |u_0(t,x)|^2$. So $u_\eps(t,x)\to u_0(t,x)$ in $L^2(\Omega)$.

We only prove Proposition \ref{prop:conVarg2}. The proof of Proposition \ref{prop:conExpg2} is similar with some simplifications.

\begin{proof}[Proof of Proposition \ref{prop:conVarg2}]
The goal is to show that for any $a,b\in \R^d,c\in \R$, as $\eps\to 0$
\begin{equation}
\E\E_B\{e^{ia\cdot \eps B_{t/\eps^2}^1+ib\cdot \eps B_{t/\eps^2}^2+ic(\eps\int_0^{t/\eps^2}V(s,\eps^\beta B_s^1)ds-\eps\int_0^{t/\eps^2}V(s,\eps^\beta B_s^2)ds)}\}\to e^{-\frac12|a|^2t-\frac12|b|^2t-c^2\sigma^2t}.
\end{equation}

We first consider the average with respect to the random environment. Let $$X_\eps(t)=\eps\int_0^{t/\eps^2}V(s,\eps^\beta B_s^1)ds-\eps\int_0^{t/\eps^2}V(s,\eps^\beta B_s^2)ds,$$ $\Delta t=\eps^{-\gamma_1}+\eps^{-\gamma_2}$, $0<\gamma_2<\gamma_1<2$ to be determined, and $N=[\frac{t}{\eps^2 \Delta t}]\sim t\eps^{\gamma_1-2}$. Define the intervals $I_k=[(k-1)\Delta t,(k-1)\Delta t+\eps^{-\gamma_1}]$ and $J_k=[(k-1)\Delta t+\eps^{-\gamma_1},k\Delta t]$ for $ k=1,\ldots, N$, we have
\begin{equation}
\begin{aligned}
X_\eps(t)=&\sum_{k=1}^N\int_{I_k}\eps (V(s,\eps^\beta B_s^1)-V(s,\eps^\beta B_s^2))ds\\
+&\sum_{k=1}^N \int_{J_k}\eps (V(s,\eps^\beta B_s^1)-V(s,\eps^\beta B_s^2))ds\\
+&\int_{N\Delta t}^{t/\eps^2}\eps (V(s,\eps^\beta B_s^1)-V(s,\eps^\beta B_s^2))ds:=(I)+(II)+(III).
\end{aligned}
\end{equation}

We show that for every realization of $B_s^1,B_s^2$, $\E\{|(II)|^2\}+\E\{|(III)|^2\}\to 0$. Take $(II)$ for example, we have
\begin{equation}
\E\{|(II)|^2\}\les \eps^2\sum_{m=1}^N\sum_{n=1}^N \int_{J_m}\int_{J_n}\sup_{x\in\R^d} |R(s-u,x)|dsdu.
\end{equation}
Since $\int_{\R}\sup_{x\in \R^d}|R(t,x)|dt<\infty$, for those terms when $m=n$, we have an order of $\eps^2N\eps^{-\gamma_2}\sim \eps^{\gamma_1-\gamma_2}\to 0$ as $\eps \to 0$. For the terms when $m\neq n$, we have $|s-u|\geq \eps^{-\gamma_1}$ if $s\in J_m, u\in J_n$, so $\sup_{x\in \R^d}|R(s-u,x)|\les \varphi^\frac12(\eps^{-\gamma_1})$, and we can use the following crude bound
\begin{equation}
\eps^2\sum_{m\neq n} \int_{J_m}\int_{J_n}\sup_{x\in\R^d} |R(s-u,x)|dsdu\leq \eps^2N^2\eps^{-2\gamma_2}\varphi^\frac12(\eps^{-\gamma_1})\les \eps^{2\gamma_1-2\gamma_2+\gamma_1\lambda-2},
\end{equation}
if $\varphi^\frac12(r)\les r^{-\lambda}$. Since $\lambda$ can be sufficiently large(e.g. $\lambda>2/\gamma_1$), we have
$$\eps^2\sum_{m\neq n} \int_{J_m}\int_{J_n}\sup_{x\in\R^d} |R(s-u,x)|dsdu\to 0$$
as $\eps\to 0$. Therefore, $\E\{|(II)|^2\}\to 0$. Similar discussion holds for $(III)$. Now we have
\begin{equation}
\E\{|e^{icX_\eps(t)}-e^{ic(I)}|\}\to 0,
\end{equation}
so
\begin{equation}
\lim_{\eps\to 0}|\E\E_B\{e^{ia\cdot \eps B_{t/\eps^2}^1+ib\cdot \eps B_{t/\eps^2}^2+icX_\eps(t)}\}-\E\E_B\{e^{ia\cdot \eps B_{t/\eps^2}^1+ib\cdot \eps B_{t/\eps^2}^2+ic(I)}\}|=0.
\end{equation}

Next, we consider $(I)=\sum_{k=1}^N\int_{I_k}\eps (V(s,\eps^\beta B_s^1)-V(s,\eps^\beta B_s^2))ds=\sum_{k=1}^N\eps^{1-\frac{\gamma_1}{2}}Y_k^\eps$, with
\begin{equation}
Y_k^\eps:=\int_{I_k}\eps^{\frac{\gamma_1}{2}} (V(s,\eps^\beta B_s^1)-V(s,\eps^\beta B_s^2))ds,
\end{equation}
Let $t_k=(k-1)\Delta t+\eps^{-\gamma_1}, k=1,\ldots, N$, then
\begin{equation}
\E\{e^{ic(I)}\}=\E\{e^{ic \sum_{k=1}^N\eps^{1-\frac{\gamma_1}{2}}Y_k^\eps}\}=\E\{\E\{e^{ic\sum_{k=1}^N\eps^{1-\frac{\gamma_1}{2}}Y_k^\eps}|\F_{t_{N-1}}\}\},
\end{equation}
and
\begin{equation}
\E\{e^{ic\sum_{k=1}^N\eps^{1-\frac{\gamma_1}{2}}Y_k^\eps}|\F_{t_{N-1}}\}=e^{ic \sum_{k=1}^{N-1}\eps^{1-\frac{\gamma_1}{2}}Y_k^\eps}\E\{e^{ic \eps^{1-\frac{\gamma_1}{2}}Y_N^\eps}|\F_{t_{N-1}}\}.
\end{equation}
When freezing the Brownian motions, $Y_N^\eps$ is $\F^{(N-1)\Delta t}$-measurable. Since $(N-1)\Delta t-t_{N-1}=\eps^{-\gamma_2}$ and $e^{ic \eps^{1-\frac{\gamma_1}{2}}Y_N^\eps}$ is uniformly bounded by $1$, we have by Lemma \ref{lem:condiMixing} that
\begin{equation}
|\E\{e^{ic \eps^{1-\frac{\gamma_1}{2}}Y_N^\eps}|\F_{t_{N-1}}\}
-\E\{
e^{ic \eps^{1-\frac{\gamma_1}{2}}Y_N^\eps}\}|\leq 2 \varphi(\eps^{-\gamma_2})\to 0.
\label{eq:useMixing}
\end{equation}
Therefore we obtain
\begin{equation}
\lim_{\eps \to 0}\left(\E\{e^{ic \sum_{k=1}^{N}\eps^{1-\frac{\gamma_1}{2}}Y_k^\eps}\}-\E\{e^{ic \sum_{k=1}^{N-1}\eps^{1-\frac{\gamma_1}{2}}Y_k^\eps}\}\E\{
e^{ic \eps^{1-\frac{\gamma_1}{2}}Y_N^\eps}\}\right)=0.
\end{equation}
Iterating the above procedure, in the end we have
\begin{equation}
\lim_{\eps\to 0}
\left(\E\{e^{ic\sum_{k=1}^N\eps^{1-\frac{\gamma_1}{2}}Y_k^\eps}\}-\prod_{k=1}^N \E\{e^{ic\eps^{1-\frac{\gamma_1}{2}}Y_k^\eps}\}\right)=0.
\end{equation}

Now we consider $\E\{e^{ic\eps^{1-\frac{\gamma_1}{2}}Y_k^\eps}\}$. By Taylor expansion, we have $|e^{ix}-1-ix+\frac12x^2|\leq C|x|^3$, so
\begin{equation}
\E\{e^{ic \eps^{1-\frac{\gamma_1}{2}}Y_k^\eps}\}=1-\frac12c^2\eps^{2-\gamma_1}\E\{(Y_k^\eps)^2\}+\eps^{2-\gamma_1}O(\eps^{1-\frac{\gamma_1}{2}}|\int_{I_k}\eps^{\frac{\gamma_1}{2}}(V(s,\eps^\beta B_s^1)-V(s,\eps^\beta B_s^2))ds|^3).
\end{equation}
Since $V$ is uniformly bounded, we have
\begin{equation}
\eps^{1-\frac{\gamma_1}{2}}|\int_{I_k}\eps^{\frac{\gamma_1}{2}}(V(s,\eps^\beta B_s^1)-V(s,\eps^\beta B_s^2))ds|^3
 \les \eps^{1-\frac{\gamma_1}{2}}\eps^{-\frac{3\gamma_1}{2}}=\eps^{1-2\gamma_1}\to 0,
 \end{equation}
if $\gamma_1<\frac12$. Then we obtain
\begin{equation}
\E\{e^{ic \eps^{1-\frac{\gamma_1}{2}}Y_k^\eps}\}=1-\frac12c^2\eps^{2-\gamma_1}\E\{(Y_k^\eps)^2\}+o(\eps^{2-\gamma_1}).
\end{equation}

Now we have
\begin{equation}
\prod_{k=1}^N\E\{e^{ic\eps^{1-\frac{\gamma_1}{2}}Y_k^\eps}\}=\prod_{k=1}^N\left(1-\frac12c^2\eps^{2-\gamma_1}\E\{(Y_k^\eps)^2\}+o(\eps^{2-\gamma_1})\right)=e^{\sum_{k=1}^N\log\left(1-\frac12c^2\eps^{2-\gamma_1}\E\{(Y_k^\eps)^2\}+o(\eps^{2-\gamma_1})\right)},
\end{equation}
and we claim that
\begin{equation}
\prod_{k=1}^N\E\{e^{ic\eps^{1-\frac{\gamma_1}{2}}Y_k^\eps}\}\to e^{-c^2\sigma^2 t}
\label{eq:conVarSigma}
\end{equation}
in probability. If this is true, we have
\begin{equation}
\begin{aligned}
\lim_{\eps\to 0}\E\E_B\{e^{ia\cdot \eps B_{t/\eps^2}^1+ib\cdot \eps B_{t/\eps^2}^2+icX_\eps(t)}\}=&\lim_{\eps\to 0}\E_B\{e^{ia\cdot \eps B_{t/\eps^2}^1+ib\cdot \eps B_{t/\eps^2}^2}\prod_{k=1}^N\E\{e^{ic\eps^{1-\frac{\gamma_1}{2}}Y_k^\eps}\}\}\\
=&\lim_{\eps\to 0}\E_B\{e^{ia\cdot \eps B_{t/\eps^2}^1+ib\cdot \eps B_{t/\eps^2}^2}e^{-c^2\sigma^2t}\}\\
=&e^{-\frac12|a|^2t-\frac12|b|^2t-c^2\sigma^2t},
\end{aligned}
\end{equation}
and the proof is complete.

To prove \eqref{eq:conVarSigma}, we consider $$\log\left(1-\frac12c^2\eps^{2-\gamma_1}\E\{(Y_k^\eps)^2\}+o(\eps^{2-\gamma_1})\right)=\log\left(1-\frac12c^2\eps^{2-\gamma_1}(\E\{(Y_k^\eps)^2\}+o(1))\right),$$
where $o(1)$ is uniformly bounded, independent of $k$, and $o(1)\to 0$ as $\eps\to 0$. In addition,
\begin{equation}
\E\{(Y_k^\eps)^2\}\les \eps^{\gamma_1}\int_{I_k^2}\sup_{x\in \R^d}|R(s-u,x)|dsdu
\end{equation}
is uniformly bounded. By Taylor expansion, we have
\begin{equation}
\begin{aligned}
 \sum_{k=1}^N\log\left(1-\frac12c^2\eps^{2-\gamma_1}(\E\{(Y_k^\eps)^2\}+o(1))\right)=&\sum_{k=1}^N\left(-\frac12c^2\eps^{2-\gamma_1}(\E\{(Y_k^\eps)^2\}+o(1))\right)\\+&\sum_{k=1}^NO\left(\left(\frac12c^2\eps^{2-\gamma_1}(\E\{(Y_k^\eps)^2\}+o(1))\right)^2\right).
\end{aligned}
\end{equation}
Clearly
$\sum_{k=1}^NO(\left(\frac12c^2\eps^{2-\gamma_1}(\E\{(Y_k^\eps)^2\}+o(1))\right)^2)\to 0$ in probability, and we only need to consider
\begin{equation}
(*)=\sum_{k=1}^N \eps^{2-\gamma_1}\E\{( Y_k^\eps)^2\},
\end{equation}
and prove $(*)\to 2\sigma^2t$ in probability. Actually, we have
\begin{equation}
(*)-2\sigma^2t =\sum_{k=1}^N\eps^{2-\gamma_1}(\E\{( Y_k^\eps)^2\}-2\sigma^2)+2\sigma^2(N\eps^{2-\gamma_1}-t).
\end{equation}
The second part goes to zero as $\eps \to 0$. For the first part,  we note that
\begin{equation}
\begin{aligned}
\E\{(Y_k^\eps)^2\}=&\eps^{\gamma_1}\int_{I_k^2}R(s-u,\eps^\beta(B_s^1-B_u^1))dsdu\\
+&\eps^{\gamma_1}\int_{I_k^2}R(s-u,\eps^\beta(B_s^2-B_u^2))dsdu\\
-&2\eps^{\gamma_1}\int_{I_k^2}R(s-u,\eps^\beta(B_s^1-B_u^2))dsdu.
\end{aligned}
\end{equation}
By Lemma \ref{lem:conVarSigma}, we have $\E_B\{(\eps^{\gamma_1}\int_{I_k^2}R(s-u,\eps^\beta(B_s^i-B_u^i))dsdu-\sigma^2)^2\}\to 0$ for $i=1,2$, and clearly it is independent of $k$. Now we only have to consider
\begin{equation}
\sum_{k=1}^N\eps^{2-\gamma_1}\eps^{\gamma_1}\int_{I_k^2}|R(s-u,\eps^\beta(B_s^1-B_u^2))|dsdu\leq \eps^2\int_0^{t/\eps^2}\int_0^{t/\eps^2}|R(s-u,\eps^\beta(B_s^1-B_u^2))|dsdu.
\end{equation}
Again by Lemma \ref{lem:conVarSigma}, $\eps^2\int_0^{t/\eps^2}\int_0^{t/\eps^2}|R(s-u,\eps^\beta(B_s^1-B_u^2))|dsdu\to 0$ in probability. The proof is complete.
\end{proof}

\section{$\alpha\in [0,2]$: spatial mixing and homogenization}
\label{sec:alphaS2}
\subsection{Kipnis-Varadhan's method}

When $\alpha\leq 2$, we use a different approach, which makes use of the mixing from Brownian motions.

By Feynman-Kac formula, the solution is written as
\begin{equation}
u_\eps(t,x)=\E_B\{f(x+B_t)\exp(i\eps^{-1}\int_0^t V(s/\eps^\alpha,(x+B_s)/\eps)ds)\}.
\end{equation}
By the scaling property of the Brownian motion and stationarity of $V$, we only need to consider
\begin{equation}
u_\eps(t,x)=\E_B\{f(x+\eps B_{t/\eps^2})\exp(i\eps\int_0^{t/\eps^2} V(\eps^{2-\alpha}s,B_s)ds)\}.
\end{equation}

Recall that $V(t,x)=\V(\tau_{(t,x)}\omega)$, define the environmental process $y^\eps_s=\tau_{(\eps^{2-\alpha}s,B_s)}\omega$ taking values in $\Omega$. By Lemma \ref{lem:enviPro}, for fixed $\eps>0$, it is a stationary Markov process, ergodic with respect to the invariant measure $\Pb$. By a straightforward calculation, the generator is $L=\eps^{2-\alpha}D_0+\frac12\sum_{k=1}^dD_k^2$. Now the Brownian motion in random scenery can be rewritten as $X_\eps(t):=\eps\int_0^{t/\eps^2}\V(y_s^\eps)ds$, i.e., an additive functional of a stationary, ergodic Markov process. The Kipnis-Varadhan's method involves constructing a corrector function $\Phi_\lambda$ and applying a martingale decomposition as follows.

Define the corrector function $\Phi_\lambda$ with $\lambda=\eps^2$ as
\begin{equation}
(\lambda-L)\Phi_\lambda=\V.
\end{equation}
By It\^o's formula, $X_\eps(t)=\eps\int_0^{t/\eps^2}\V(y_s^\eps)ds=R_t^\eps+M_t^\eps$ with
\begin{eqnarray}
R_t^\eps&=&\eps\int_0^{t/\eps^2}\lambda\Phi_\lambda(y_s^\eps)ds-\eps\Phi_\lambda(y_{t/\eps^2}^\eps)+\eps\Phi_\lambda(y_0^\eps),\\
M_t^\eps&=&\sum_{k=1}^d\eps\int_0^{t/\eps^2}D_k\Phi_\lambda(y_s^\eps)dB_s^k.
\end{eqnarray}

To show the distribution of $X_\eps(t)$ is close to a normal distribution, the idea is to prove $R_t^\eps$ is small and the martingale $M_t^\eps$ is close to a Brownian motion. We first have the following lemma.
\begin{lemma}
\begin{equation}
\lambda\langle \Phi_\lambda,\Phi_\lambda\rangle\to 0
\end{equation}
as $\lambda\to 0$.
\label{lem:smallCorrector}
\end{lemma}

\begin{proof}
By spectral representation, $\Phi_\lambda$ is written as
\begin{equation}
\Phi_\lambda=\int_{\R^{d+1}}\frac{1}{\lambda+\frac12|\xi|^2-i\xi_0\eps^{2-\alpha}}U(d\xi_0,d\xi)\V,
\end{equation}
so
\begin{equation}
\lambda\langle \Phi_\lambda,\Phi_\lambda\rangle=\frac{1}{(2\pi)^{d+1}}\int_{\R^{d+1}}\frac{\lambda\hat{R}(\xi_0,\xi)}{(\lambda+\frac12|\xi|^2)^2+\eps^{4-2\alpha}\xi_0^2}d\xi_0d\xi.
\end{equation}
Clearly $\frac{\lambda\hat{R}(\xi_0,\xi)}{(\lambda+\frac12|\xi|^2)^2+\eps^{4-2\alpha}\xi_0^2}\les \frac{\hat{R}(\xi_0,\xi)}{|\xi|^2}$. Since $\hat{R}$ is bounded and integrable, and $\frac{1}{|\xi|^2}$ is integrable around the origin when $d\geq 3$, by the dominated convergence theorem, the proof is complete.
\end{proof}

By Lemma \ref{lem:smallCorrector}, we have
\begin{equation}
\E\E_B\{|R_t^\eps|^2\}\les \lambda\langle \Phi_\lambda,\Phi_\lambda\rangle\to 0
\end{equation}
as $\eps \to 0$.

To deal with the martingale $M_t^\eps$, we distinguish between $\alpha=2$ and $\alpha<2$.

\subsection{$\alpha=2$: ergodicity}

Let $\eta_k=\int_{\R^{d+1}}\frac{i\xi_k}{\frac12|\xi|^2-i\xi_0}U(d\xi_0,d\xi)\V\in L^2(\Omega)$, $k=1,\ldots,d$, we have
\begin{lemma}
$D_k\Phi_\lambda\to \eta_k$ in $L^2(\Omega)$, $k=1,\ldots,d$.
\label{lem:conGradient}
\end{lemma}

\begin{proof}
Since
\begin{equation}
D_k\Phi_\lambda=\int_{\R^{d+1}}\frac{i\xi_k}{\lambda+\frac12|\xi|^2-i\xi_0}U(d\xi_0,d\xi)\V,
\end{equation}
we have
\begin{equation}
\begin{aligned}
\| D_k\Phi_\lambda-\eta_k\|^2=&\frac{1}{(2\pi)^{d+1}}\int_{\R^{d+1}}\frac{\lambda^2\xi_k^2\hat{R}(\xi_0,\xi)}{((\lambda+\frac12|\xi|^2)^2+\xi_0^2)(\frac14|\xi|^4+\xi_0^2)}d\xi_0d\xi\\
\les &\int_{\R^{d+1}}\frac{\lambda^2}{(\lambda+\frac12|\xi|^2)^2+\xi_0^2}\frac{|\xi|^2\hat{R}(\xi_0,\xi)}{\frac14|\xi|^4+|\xi_0|^2}d\xi_0d\xi\to 0
\end{aligned}
\end{equation}
by the dominated convergence theorem.
\end{proof}

Let $M_t=\sum_{k=1}^d \eps\int_0^{t/\eps^2}\eta_k(y_s^\eps)dB_s^k$ and $$\sigma^2=\sum_{k=1}^d\|\eta_k\|^2=\frac{1}{(2\pi)^{d+1}}\int_{\R^{d+1}}\frac{|\xi|^2\hat{R}(\xi_0,\xi)}{\frac14|\xi|^4+\xi_0^2}d\xi_0d\xi=2\int_0^\infty \E_B\{R(t,B_t)\}dt,$$
and since $u_0$ solves $\partial_t u_0(t,x)=\frac12\Delta u_0(t,x)-\frac12\sigma^2 u_0(t,x)$, i.e.,
\begin{equation}
u_0(t,x)=\E_B\{f(x+\eps B_{t/\eps^2})\exp(-\frac12\sigma^2 t)\},
\end{equation}
we have the decomposition of the error
\begin{equation}
\begin{aligned}
u_\eps(t,x)-u_0(t,x)=&\E_B\{f(x+\eps B_{t/\eps^2})\exp(iX_\eps(t))\}-\E_B\{f(x+\eps B_{t/\eps^2})\exp(i M_t^\eps)\}\\
+&\E_B\{f(x+\eps B_{t/\eps^2})\exp(iM_t^\eps)\}-\E_B\{f(x+\eps B_{t/\eps^2})\exp(iM_t)\}\\
+&\E_B\{f(x+\eps B_{t/\eps^2})\exp(i M_t)\}-\E_B\{f(x+\eps B_{t/\eps^2})\exp(-\frac12\sigma^2 t)\}\\
:=&(I)+(II)+(III).
\end{aligned}
\end{equation}
Clearly $\E\{|(I)|\}\les \E\E_B\{|R_t^\eps|\}\les \sqrt{\lambda\langle\Phi_\lambda,\Phi_\lambda\rangle}\to 0$, and
\begin{equation}
\E\{|(II)|\}\les \sqrt{\E\E_B\{|M_t^\eps-M_t|^2\}}=\sqrt{\sum_{k=1}^d \|D_k\Phi_\lambda-\eta_k\|^2t}\to 0.
\end{equation}

For $(III)$, we note that $M_t=\sum_{k=1}^d \eps\int_0^{t/\eps^2}\eta_k(y_s^\eps)dB_s^k$ is a square-integrable martingale for almost every $\omega\in \Omega$, and when $\alpha=2$, $y_s^\eps=\tau_{(s,B_s)}\omega$ is ergodic. By martingale central limit theorem, ergodicity of $y_s^\eps$, and the fact that $\E\{\eta_k\}=0,k=1,\ldots,d$, we obtain that for almost every $\omega$
\begin{equation}
(\eps B_{t/\eps^2}, M_t)\Rightarrow (W^1_t, \sigma W^2_t),
\label{eq:wkConEr}
\end{equation}
with $W^1,W^2$ independent Brownian motions. Thus $(III)\to 0$ almost surely. By the dominated convergence theorem we have $\E\{|(III)|\}\to 0$ as $\eps \to 0$.

To summarize, $\E\{|u_\eps(t,x)-u_0(t,x)|\}\to 0$ as $\eps \to 0$. The proof of the case $\alpha=2$ is complete.

\begin{remark}
It is worth noting that for the case $\alpha=2$, the mixing property in Assumption \ref{ass:mixing} is not used. All we need is the ergodicity and certain integrability condition of $\hat{R}(\xi_0,\xi)$.
\end{remark}


\subsection{$\alpha\in [0,2)$: quantitative martingale central limit theorem}

In this regime, $y_s^\eps=\tau_{(\eps^{2-\alpha}s,B_s)}\omega$ is $\eps-$dependent, so an ergodicity does not seem to establish \eqref{eq:wkConEr}. We apply a quantitative martingale central limit theorem instead. A fourth moment estimation appears in the proof, for which we use the mixing property of $V$.

We define $\sigma_\lambda^2=\sum_{k=1}^d \|D_k\Phi_\lambda\|^2$, and \begin{equation}
\sigma^2=\frac{4}{(2\pi)^{d+1}}\int_{\R^{d+1}}\frac{\hat{R}(\xi_0,\xi)}{|\xi|^{2}}d\xi_0d\xi=2\int_0^\infty \E_B\{R(0,B_t)\}dt,
\end{equation}
the following lemma holds.

\begin{lemma}
$\sigma_\lambda^2\to \sigma^2$ as $\lambda\to 0$.
\end{lemma}

\begin{proof}
Since \begin{equation}
D_k\Phi_\lambda=\int_{\R^{d+1}}\frac{i\xi_k}{\lambda+\frac12|\xi|^2-i\xi_0\eps^{2-\alpha}}U(d\xi_0,d\xi)\V,
\end{equation}
we have
\begin{equation}
\sigma_\lambda^2=\frac{1}{(2\pi)^{d+1}}\int_{\R^{d+1}}\frac{|\xi|^2\hat{R}(\xi_0,\xi)}{(\lambda+\frac12|\xi|^2)^2+\xi_0^2\eps^{4-2\alpha}}d\xi_0d\xi.
\end{equation}
Since $\alpha<2$, by the dominated convergence theorem, the proof is complete.
\end{proof}

Now we decompose the error as
\begin{equation}
\begin{aligned}
u_\eps(t,x)-u_0(t,x)=&\E_B\{f(x+\eps B_{t/\eps^2})\exp(iX_\eps(t))\}-\E_B\{f(x+\eps B_{t/\eps^2})\exp(i M_t^\eps)\}\\
+&\E_B\{f(x+\eps B_{t/\eps^2})\exp(iM_t^\eps)\}-\E_B\{f(x+\eps B_{t/\eps^2})\exp(-\frac12\sigma_\lambda^2t)\}\\
+&\E_B\{f(x+\eps B_{t/\eps^2})\exp(-\frac12\sigma_\lambda^2)\}-\E_B\{f(x+\eps B_{t/\eps^2})\exp(-\frac12\sigma^2 t)\}\\
:=&(I)+(II)+(III).
\end{aligned}
\end{equation}
By the same discussion as before, we have $\E\{|(I)|\}\les \E\E_B\{|R_t^\eps|\}\les \sqrt{\lambda\langle\Phi_\lambda,\Phi_\lambda\rangle}\to 0$, and $\E\{|(III)|\}\les |\sigma_\lambda^2-\sigma^2|\to 0$.

For $(II)$, we rewrite it in Fourier domain as
\begin{equation}
\begin{aligned}
(II)=&\E_B\{f(x+\eps B_{t/\eps^2})\exp(iM_t^\eps)\}-\E_B\{f(x+\eps B_{t/\eps^2})\exp(-\frac12\sigma_\lambda^2t)\}\\
=&\frac{1}{(2\pi)^d}\int_{\R^d}\hat{f}(\xi)e^{i\xi\cdot x}\E_B\{e^{i\xi\cdot \eps B_{t/\eps^2}+iM_t^\eps}-e^{-\frac12(|\xi|^2+\sigma_\lambda^2)t}\}d\xi.\\
\end{aligned}
\end{equation}
$\xi\cdot \eps B_{t/\eps^2}+M_t^\eps=\sum_{k=1}^d \eps\int_0^{t/\eps^2}(\xi_k+D_k\Phi_\lambda(y_s^\eps))dB_s^k$ is a continuous and square-integrable martingale for almost every $\omega\in \Omega$, so the estimation of $(II)$ boils down to a control of the Wasserstein distance between the martingale and a Brownian motion. By quantitative martingale central limit theorem \cite[Theorem 3.2]{mourrat2012kantorovich}, we obtain for some constant $C$ that
\begin{equation}
|\E_B\{e^{i\xi\cdot \eps B_{t/\eps^2}+iM_t^\eps}-e^{-\frac12(|\xi|^2+\sigma_\lambda^2)t}\}|\leq C\E_B\{|\sum_{k=1}^d \eps^2\int_0^{t/\eps^2}(\xi_k+D_k\Phi_\lambda(y_s^\eps))^2ds-(|\xi|^2+\sigma_\lambda^2)t|\}.
\end{equation}

Let $\sum_{k=1}^d \eps^2\int_0^{t/\eps^2}(\xi_k+D_k\Phi_\lambda(y_s^\eps))^2ds-(|\xi|^2+\sigma_\lambda^2)t=(i)+(ii)$ with
\begin{eqnarray}
(i)&=&2\sum_{k=1}^d\xi_k \eps^2\int_0^{t/\eps^2} D_k\Phi_\lambda(y_s^\eps)ds,\\
(ii)&=&\eps^2\int_0^{t/\eps^2}\left(\sum_{k=1}^d (D_k\Phi_\lambda(y_s^\eps))^2-\sigma_\lambda^2\right)ds.
\end{eqnarray}
We will show that $\E\E_B\{|\eps^2\int_0^{t/\eps^2}D_k\Phi_\lambda(y_s^\eps)ds|\}\to 0$ and $\E\E_B\{|(ii)|\}\to 0$ as $\eps \to 0$ in Lemma \ref{lem:alphal21} and \ref{lem:alphal22} below, which implies $\E\{|(II)|\}\to 0$ by the dominated convergence theorem if we assume $|\hat{f}(\xi)||\xi|\in L^1$. Therefore $\E\{|u_\eps(t,x)-u_0(t,x)|\}\to 0$, and the proof of the case $\alpha<2$ is complete.

\begin{lemma}
$\E\E_B\{|\eps^2\int_0^{t/\eps^2}D_k\Phi_\lambda(y_s^\eps)ds|\}\to 0$, $k=1,\ldots,d$.
\label{lem:alphal21}
\end{lemma}

\begin{proof}
We recall that
\begin{equation}
D_k\Phi_\lambda=\int_{\R^{d+1}}\frac{i\xi_k}{\lambda+\frac12|\xi|^2-i\xi_0\eps^{2-\alpha}}U(d\xi_0,d\xi)\V,
\end{equation}
so
\begin{equation}
\langle T_{(t,x)}D_k\Phi_\lambda,D_k\Phi_\lambda\rangle=\frac{1}{(2\pi)^{d+1}}\int_{\R^{d+1}}\frac{e^{i\xi_0\cdot t}e^{i\xi\cdot x}\xi_k^2\hat{R}(\xi_0,\xi)}{(\lambda+\frac12|\xi|^2)^2+\xi_0^2\eps^{4-2\alpha}}d\xi_0d\xi.
\end{equation}
This leads to
\begin{equation}
\E\E_B\{D_k\Phi_\lambda(y_s^\eps)D_k\Phi_\lambda(y_u^\eps)\}=\frac{1}{(2\pi)^{d+1}}\int_{\R^{d+1}}\frac{e^{i\xi_0\cdot \eps^{2-\alpha}(s-u)}e^{-\frac12|\xi|^2|s-u|}\xi_k^2\hat{R}(\xi_0,\xi)}{(\lambda+\frac12|\xi|^2)^2+\xi_0^2\eps^{4-2\alpha}}d\xi_0d\xi.
\end{equation}
Therefore,
\begin{equation}
\begin{aligned}
(\E\E_B\{|\eps^2\int_0^{t/\eps^2}D_k\Phi_\lambda(y_s^\eps)ds|\})^2\leq& \E\E_B\{\eps^4\int_0^{t/\eps^2}\int_0^{t/\eps^2}D_k\Phi_\lambda(y_s^\eps)D_k\Phi_\lambda(y_u^\eps)dsdu\}\\
\les &\eps^4\int_0^{t/\eps^2}\int_0^{t/\eps^2}\int_{\R^{d+1}}\frac{e^{-\frac12|\xi|^2|s-u|}\hat{R}(\xi_0,\xi)}{|\xi|^2}d\xi_0d\xi dsdu\\
=&\int_0^{t}\int_0^{t}\int_{\R^{d+1}}\frac{e^{-\frac{1}{2\eps^2}|\xi|^2|s-u|}\hat{R}(\xi_0,\xi)}{|\xi|^2}d\xi_0d\xi dsdu\to 0
\end{aligned}
\end{equation}
by the dominated convergence theorem.
\end{proof}

\begin{lemma}
$\E\E_B\{|(ii)|\}\to 0$.
\label{lem:alphal22}
\end{lemma}

\begin{proof}
Define $\mathcal{Z}_\lambda(t,x)=\sum_{k=1}^d (D_k\Phi_\lambda(\tau_{(\eps^{2-\alpha}t,x)}\omega))^2-\sigma_\lambda^2$, which is zero-mean and stationary. Then $(ii)=\eps^2\int_0^{t/\eps^2}\mathcal{Z}_\lambda(s,B_s)ds$. Denote the covariance function of $\mathcal{Z}_\lambda$ by $\mathcal{R}_\lambda(t,x)$, Lemma \ref{lem:VarBMRS} implies
\begin{equation}
\E\E_B\{|(ii)|^2\}\les \eps^2\int_{\R^d}\frac{\sup_t|\mathcal{R}_\lambda(t,x)|}{|x|^{d-2}}dx,
\end{equation}
so we only need to estimate $\sup_t |\mathcal{R}_\lambda(t,x)|$.

Clearly,
\begin{equation}
\mathcal{R}_\lambda(t,x)=\sum_{m,n=1}^d \E\{(D_m\Phi_\lambda(\tau_{(\eps^{2-\alpha}t,x)}\omega))^2(D_n\Phi_\lambda(\tau_{(0,0)}\omega))^2\}-\sigma_\lambda^4.
\end{equation}
Let $G_\lambda(t,x)$ be the Green's function of $\lambda-\eps^{2-\alpha}\partial_t-\frac12\Delta$, it is straightforward to check that
\begin{equation}
G_\lambda(t,x)=\eps^{\alpha-2}e^{\eps^\alpha t}q_{-t\eps^{\alpha-2}}(x)1_{t<0},
\end{equation}
so $D_k\Phi_\lambda(\tau_{(t,x)}\omega)=\int_{\R^{d+1}}\partial_{x_k}G_\lambda(t-s,x-y)V(s,y)dsdy$, and we obtain
\begin{equation}
\begin{aligned}
&\E\{(D_m\Phi_\lambda(\tau_{(\eps^{2-\alpha}t,x)}\omega))^2(D_n\Phi_\lambda(\tau_{(0,0)}\omega))^2\}\\
=&\int_{\R^{4d+4}}\prod_{i=1}^2\partial_{x_m}G_\lambda(\eps^{2-\alpha}t-s_i,x-y_i)\prod_{i=3}^4\partial_{x_n}G_\lambda(-s_i,-y_i)\E\{\prod_{i=1}^4V(s_i,y_i)\}dsdy.
\end{aligned}
\end{equation}
By Lemma \ref{lem:4Moment}, we have
\begin{equation}
\begin{aligned}
&|\E\{\prod_{i=1}^4V(s_i,y_i)\}-R(s_1-s_2,y_1-y_2)R(s_3-s_4,y_3-y_4)|\\
\leq &\Psi(s_1-s_3,y_1-y_3)\Psi(s_2-s_4,y_2-y_4)+
\Psi(s_1-s_4,y_1-y_4)\Psi(s_2-s_3,y_2-y_3)\\
\leq &g(y_1-y_3)g(y_2-y_4)+g(y_1-y_4)g(y_2-y_3)
\end{aligned}
\end{equation}
for $g(x)=\sup_t \Psi(t,x)$. In addition, we have $\int_\R |\partial_{x_k} G_\lambda(s,x)|ds=\int_0^\infty e^{-\lambda t}q_t(x)\frac{|x_k|}{t}dt=|\partial_{x_k}\mathcal{G}_\lambda(x)|$, where $\mathcal{G}_\lambda(x)=\int_0^\infty e^{-\lambda t}q_t(x)dt$ is the Green's function of $\lambda-\frac12\Delta$. Therefore, by the fact that
\begin{equation}
\sigma_\lambda^4=\sum_{m,n=1}^d\int_{\R^{4d+4}}\prod_{i=1}^2\partial_{x_m}G_\lambda(\eps^{2-\alpha}t-s_i,x-y_i)\prod_{i=3}^4\partial_{x_n}G_\lambda(-s_i,-y_i)R(s_1-s_2,y_1-y_2)R(s_3-s_4,y_3-y_4)\}dsdy
\end{equation}
and
\begin{equation}
\begin{aligned}
&\int_{\R^{4d+4}}\prod_{i=1}^2|\partial_{x_m}\G_\lambda(x-y_i)|\prod_{i=3}^4|\partial_{x_n}\G_\lambda(-y_i)|
\left(g(y_1-y_3)g(y_2-y_4)+g(y_1-y_4)g(y_2-y_3)\right)dy\\
\les &\left(\int_{\R^{2d}}|\partial_{x_m}\G_\lambda(y)||\partial_{x_n}\G_\lambda(z)|g(x-y-z)dydz\right)^2,
\end{aligned}
\end{equation}
we obtain
\begin{equation}
\sup_t|\mathcal{R}_\lambda(t,x)|\les \left(\int_{\R^{2d}}\frac{e^{-c\sqrt{\lambda}|y|}}{|y|^{d-1}}\frac{e^{-c\sqrt{\lambda}|z|}}{|z|^{d-1}}g(x-y-z)dydz\right)^2,
\end{equation}
where we have also used the fact $|\partial_{x_k} \G_\lambda(y)|\leq Ce^{-c\sqrt{\lambda}|x|}|x|^{1-d}$ for some $c,C>0$. Since $g(x)\leq C_\beta (1\wedge |x|^{-\beta})$ for any $\beta>0$, by \cite[estimation of (3.40)]{gu2013weak}
\begin{equation}
\sup_t|\mathcal{R}_\lambda(t,x)|\les \lambda^{\frac{d}{2}-1}e^{-c\sqrt{\lambda}|x|}+1\wedge \frac{e^{-c\sqrt{\lambda}|x|}}{|x|^{d-2}}+1\wedge \frac{1}{|x|^\beta}
\end{equation}
for some constant $c>0$ and $\beta>0$ sufficiently large. By direct calculation, we have \begin{equation}
\E\E_B\{|(ii)|^2\}\les \eps^2\int_{\R^d}\frac{\sup_t|\mathcal{R}_\lambda(t,x)|}{|x|^{d-2}}dx\to 0.
\end{equation}
 The proof is complete.
\end{proof}

\begin{remark}
By a similar discussion as in \cite{gu2013weak}, we can actually establish an error estimate for the case $\alpha\in [0,2)$, which we do not present here.
\end{remark}

\section{$\alpha=\infty$: temporal mixing and convergence to SPDE}
\label{sec:alphaInfty}

In this regime, $V_\eps(t,x)=\frac{1}{\sqrt{\eps}}V(\frac{t}{\eps},x)$, so heuristically it approaches a white noise in the temporal direction. We define a formally written random variable $\int_0^t\dot{W}(t-s,x+B_s)ds$, i.e., a Brownian motion in Gaussian noise, by standard mollification argument.

Let $\RR(x)=\int_\R R(t,x)dt$ and define the Gaussian noise $W(dt,dx)$ on $\Omega$ with a formal covariance structure
\begin{equation}
\E\{\dot{W}(t,x)\dot{W}(s,y)\}=\delta(t-s)\RR(x-y).
\end{equation}
Let $B_t$ be an independent Brownian motion.

We pick a mollifier $\phi_\eps(t)q_\eps(x)$ and define the stationary random field
\begin{equation}
W_\eps(t,x)=\int_{\R^{d+1}}\phi_\eps(t-s)q_\eps(x-y)W(ds,dy),
\end{equation}
where $\phi_\eps(t)=\frac{1}{\eps}1_{[-\eps,0]}(t)$ and $q_\eps(x)$ is the heat kernel with variance $\eps$. $W_\eps(t,x)$ is a well-defined Wiener integral and clearly\begin{equation}
\E\{W_\eps(t,x)W_\eps(s,y)\}=\frac{\eps-|t-s|}{\eps^2}1_{|t-s|<\eps}\frac{1}{(2\pi)^d}\int_{\R^d}e^{-|\xi|^2\eps}\hat{\RR}(\xi)e^{i\xi\cdot (x-y)}d\xi.
\label{eq:varWeps}
\end{equation}

Since $\E\{W_\eps(t,x)^2\}<\infty$, $\int_0^tW_\eps(t-s,x+B_s)ds$ is well-defined for every realization of $B_s$.

\begin{proposition}
For almost every realization of $B_s$, $\int_0^t W_\eps(t-s,x+B_s)ds$ is a Cauchy sequence in $L^2(\Omega)$, whose limit is denoted as $\int_0^t\dot{W}(t-s,x+B_s)ds\sim N(0,\RR(0)t)$. 
\label{prop:cauchy}
\end{proposition}

\begin{proof}
We will show that for almost every realization of $B_s$,
\begin{equation}
\lim_{\eps_1,\eps_2\to 0}\E\{\int_0^t W_{\eps_1}(t-s,x+B_s)ds\int_0^t W_{\eps_2}(t-s,x+B_s)ds\}=\RR(0)t,
\end{equation} so $\int_0^t W_\eps(t-s,x+B_s)ds$ is a Cauchy sequence in $L^2(\Omega)$.

By direct calculation, we have
\begin{equation}
\begin{aligned}
&\E\{\int_0^t W_{\eps_1}(t-s,x+B_s)ds\int_0^t W_{\eps_2}(t-s,x+B_s)ds\}\\
=&\int_0^t\int_0^tdsdu \int_{\R^{2d+1}}\phi_{\eps_1}(t-s-r)q_{\eps_1}(x+B_s-y) \phi_{\eps_2}(t-u-r)q_{\eps_2}(x+B_u-z)\RR(y-z)drdydz\\
=&\int_0^t\int_0^tdsdu \int_{\R^{2d+1}}\phi_{\eps_1}(s-r) \phi_{\eps_2}(u-r)q_{\eps_1}(y)q_{\eps_2}(z)\RR(y-z+B_{t-s}-B_{t-u})drdydz\\
=&\frac{1}{(2\pi)^d}\int_{\R^{d}}e^{-\frac12|\xi|^2(\eps_1+\eps_2)}\hat{\RR}(\xi)\left(\int_\R\int_0^t\int_0^t\phi_{\eps_1}(s-r) \phi_{\eps_2}(u-r)e^{i\xi\cdot (B_{t-s}-B_{t-u})}dsdudr\right)d\xi.
\end{aligned}
\end{equation}
For fixed $\xi\in \R^d$, we consider \begin{equation}
\begin{aligned}
&\int_\R\int_0^t\int_0^t\phi_{\eps_1}(s-r) \phi_{\eps_2}(u-r)e^{i\xi\cdot (B_{t-s}-B_{t-u})}dsdudr\\
=&\frac{1}{\eps_1\eps_2}\int_\R\int_0^t\int_0^t 1_{r-\eps_1<s<r}1_{r-\eps_2<u<r}e^{i\xi\cdot (B_{t-s}-B_{t-u})}dsdudr\\
=&\int_{0}^{t+(\eps_1\vee\eps_2)}\left(\frac{1}{\eps_1\eps_2}\int_0^t\int_0^t 1_{r-\eps_1<s<r}1_{r-\eps_2<u<r}e^{i\xi\cdot (B_{t-s}-B_{t-u})}dsdu\right)dr.
\end{aligned}
\end{equation}
For almost every realization, $B_s$ is continuous in $[0,t]$, so we have as $\eps_1,\eps_2\to 0$, $e^{i\xi\cdot (B_{t-s}-B_{t-u})}\to 1$ almost surely. Thus for fixed $r\in (0,t)$, $\frac{1}{\eps_1\eps_2}\int_0^t\int_0^t 1_{r-\eps_1<s<r}1_{r-\eps_2<u<r}e^{i\xi\cdot (B_{t-s}-B_{t-u})}dsdu\to 1$ almost surely, which implies $\int_\R\int_0^t\int_0^t\phi_{\eps_1}(s-r) \phi_{\eps_2}(u-r)e^{i\xi\cdot (B_{t-s}-B_{t-u})}dsdudr\to t$ almost surely. By the dominated convergence theorem, we have
\begin{equation}
\E\{\int_0^t W_{\eps_1}(t-s,x+B_s)ds\int_0^t W_{\eps_2}(t-s,x+B_s)ds\}\to \RR(0)t.
\end{equation}
Therefore,  for almost every realization of $B_s$, we can define a random variable $\int_0^t\dot{W}(t-s,x+B_s)ds$ by $\int_0^t\dot{W}(t-s,x+B_s)ds=\lim_{\eps\to0}\int_0^t W_\eps(t-s,x+B_s)ds$. For fixed $B_s$, $\int_0^t W_\eps(t-s,x+B_s)ds$ is a zero-mean Gaussian, and by the above calculation, we obtain $\E\{(\int_0^t W_\eps(t-s,x+B_s)ds)^2\}\to \RR(0)t$. So $\int_0^t\dot{W}(t-s,x+B_s)ds\sim N(0,\RR(0)t)$ for almost every realization of $B_s$.

 The proof is complete.
\end{proof}

\begin{remark}
By the same proof as in Proposition \ref{prop:cauchy}, we can define $\int_0^t\dot{W}(t-s,x+f_s)ds, \int_0^t\dot{W}(t-s,x+g_s)ds$ for any two continuous paths $f_s,g_s$, and obtain that
\begin{equation}
\E\{\int_0^t\dot{W}(t-s,x+f_s)ds\int_0^t\dot{W}(t-s,x+g_s)ds\}=\int_0^t\RR(f_s-g_s)ds.
\end{equation}
In particular, we can choose $f_s=B_s$ and $g_s=W_s$ for independent Brownian motions $B,W$.
\label{re:cauchy}
\end{remark}


The solution to the SPDE with multiplicative noise in the Stratonovich's sense
\begin{equation}
\partial_tu_0=\frac12\Delta u_0+i \dot{W}\circ u_0,
\label{eq:limitSPDE}
\end{equation} and initial condition $u_0=f$ is then defined by Feynman-Kac formula as
\begin{equation}
u_0(t,x)=\E_B\{f(x+B_t)\exp(i\int_0^t\dot{W}(t-s,x+B_s)ds)\}.
\label{eq:FKSoSPDE}
\end{equation}
\eqref{eq:limitSPDE} is written in the It\^o's form as
\begin{equation}
\partial_t u_0(t,x)=\frac12\Delta u_0(t,x)-\frac12\RR(0)u_0(t,x)+i\dot{W}(t,x)u_0(t,x).
\label{eq:limitSPDEIto}
\end{equation}
By \cite[Theorem 3.1]{hu2012nonlinear}, the solution given by \eqref{eq:FKSoSPDE} is a mild solution to \eqref{eq:limitSPDEIto}.

Our goal is to show that for fixed $(t,x)$, $u_\eps(t,x)=\E_B\{f(x+B_t)\exp(i\frac{1}{\sqrt{\eps}}\int_0^tV(\frac{s}{\eps},x+B_s)ds)\}$ converges in distribution to $u_0(t,x)$. Since $u_\eps(t,x), u_0(t,x)$ are both bounded, it suffices to prove the convergence of moments.
\begin{proposition}
 For any $N_1,N_2\in \mathbb{N}$, as $\eps \to 0$,
\begin{equation}
\E\{u_\eps(t,x)^{N_1}\overline{u_\eps(t,x)^{N_2}}\}\to \E\{u_0(t,x)^{N_1}\overline{u_0(t,x)^{N_2}}\}.
\end{equation}
\label{prop:conMoment}
\end{proposition}

\begin{proof}
The proof is similar with Proposition \ref{prop:conVarg2}. Wherever the argument can be directly copied here, we do not present the details.

First, by the scaling property of Brownian motion, a change of parameter $\sqrt{\eps}\mapsto \eps$, and the stationarity of $V$, we have
\begin{equation}
\begin{aligned}
&\E\{u_\eps(t,x)^{N_1}\overline{u_\eps(t,x)^{N_2}}\}\\
=&\E\E_B\{\prod_{j=1}^{N_1+N_2}f(x+\eps B_{t/\eps^2}^j)\exp(i\sum_{j=1}^{N_1}\eps\int_0^{t/\eps^2}V(s,\eps B_s^j)ds-i\sum_{j=N_1+1}^{N_1+N_2}\eps\int_0^{t/\eps^2}V(s,\eps B_s^j)ds)\},
\end{aligned}
\end{equation}
where $B^j,j=1,\ldots, N_1+N_2$ are independent Brownian motions. On the other hand, by the Feynman-Kac representation of $u_0$, we have
\begin{equation*}
\begin{aligned}
&\E\{u_0(t,x)^{N_1}\overline{u_0(t,x)^{N_2}}\}\\
=&\E\E_B\{\prod_{j=1}^{N_1+N_2}f(x+B_{t}^j)\exp(i\sum_{j=1}^{N_1}\int_0^{t}\dot{W}(t-s,x+B_s^j)ds-i\sum_{j=N_1+1}^{N_1+N_2}\int_0^{t}\dot{W}(t-s,x+B_s^j)ds)\}.
\end{aligned}
\end{equation*}
Therefore, it boils down to show in the annealed sense that
\begin{equation}
\begin{aligned}
&(\eps B_{t/\eps^2}^1,\ldots,\eps B_{t/\eps^2}^{N_1+N_2}, \eps\int_0^{t/\eps^2}V(s,\eps B_s^1)ds, \ldots, \eps\int_0^{t/\eps^2}V(s,\eps B_s^{N_1+N_2})ds)\\
\Rightarrow & (B_t^1,\ldots, B_t^{N_1+N_2}, \int_0^t\dot{W}(t-s,x+B_s^1)ds,\ldots, \int_0^t\dot{W}(t-s,x+B_s^{N_1+N_2})ds).
\end{aligned}
\end{equation}
By Proposition \ref{prop:cauchy} and Remark \ref{re:cauchy}, we have for any $a_j\in \R^d, b_j\in \R, j=1,\ldots, N_1+N_2$ that
\begin{equation}
\begin{aligned}
&\E\E_B\{\exp(i\sum_{j=1}^{N_1+N_2}a_j\cdot B_t^j+i\sum_{j=1}^{N_1+N_2}b_j\int_0^t\dot{W}(t-s,x+B_s^j)ds)\}\\
=&\E_B\{\exp(i\sum_{j=1}^{N_1+N_2}a_j\cdot B_t^j)\exp(-\frac12\sum_{j_1,j_2=1}^{N_1+N_2}b_{j_1}b_{j_2}\int_0^t\RR(B_s^{j_1}-B_s^{j_2})ds)\}.
\label{eq:limitChr}
\end{aligned}
\end{equation}
Let $X_\eps(t):=\sum_{j=1}^{N_1+N_2}a_j\cdot \eps B_{t/\eps^2}^j+\sum_{j=1}^{N_1+N_2}b_j\eps \int_0^{t/\eps^2}V(s,\eps B_s^j)ds$, and we prove $\E\E_B\{e^{iX_\eps(t)}\}$ converges to the RHS of \eqref{eq:limitChr}.

Let $\Delta t=\eps^{-\gamma_1}+\eps^{-\gamma_2}$, $0<\gamma_2<\gamma_1<2$ to be determined, and $N=[\frac{t}{\eps^2 \Delta t}]\sim t\eps^{\gamma_1-2}$. Define the intervals $I_k=[(k-1)\Delta t,(k-1)\Delta t+\eps^{-\gamma_1}]$ and $J_k=[(k-1)\Delta t+\eps^{-\gamma_1},k\Delta t]$ for $ k=1,\ldots, N$. By the same discussion, we have
\begin{equation*}
\lim_{\eps\to 0}\left(\E\E_B\{e^{iX_\eps(t)}\}-\E\E_B\{\exp(i\sum_{j=1}^{N_1+N_2}a_j\cdot \eps B_{t/\eps^2}^j+i\sum_{j=1}^{N_1+N_2}b_j\eps \sum_{k=1}^N\int_{I_k}V(s,\eps B_s^j)ds)\}\right)=0.
\end{equation*}

Then we consider $\E\{\exp(i\sum_{j=1}^{N_1+N_2}b_j\eps \sum_{k=1}^N\int_{I_k}V(s,\eps B_s^j)ds)\}$. By the same discussion, we have for every realization of $B_s^j$ that
\begin{equation*}
\lim_{\eps\to 0}\left(\E\{\exp(i\sum_{j=1}^{N_1+N_2}b_j\eps \sum_{k=1}^N\int_{I_k}V(s,\eps B_s^j)ds)\}-\prod_{k=1}^N\E\{\exp(i\sum_{j=1}^{N_1+N_2}b_j\eps \int_{I_k}V(s,\eps B_s^j)ds)\}\right)=0.
\end{equation*}

Since
\begin{equation*}
\prod_{k=1}^N\E\{\exp(i\sum_{j=1}^{N_1+N_2}b_j\eps \int_{I_k}V(s,\eps B_s^j)ds)\}=\exp(\sum_{k=1}^N\log\E\{\exp(i\sum_{j=1}^{N_1+N_2}b_j\eps \int_{I_k}V(s,\eps B_s^j)ds)\}),
\end{equation*}
by the same discussion,  we have
\begin{equation*}
\log\E\{\exp(i\sum_{j=1}^{N_1+N_2}b_j\eps \int_{I_k}V(s,\eps B_s^j)ds)\}=-\frac12\sum_{m,n=1}^{N_1+N_2}b_mb_n\eps^2\int_{I_k^2}R(s-u,\eps B_s^m-\eps B_u^n)dsdu+o(\eps^{2-\gamma_1}),
\end{equation*}
where $\frac{o(\eps^{2-\gamma_1})}{\eps^{2-\gamma_1}}$ is independent of $k$, uniformly bounded and goes to zero as $\eps\to 0$. Thus we obtain
\begin{equation}
\begin{aligned}
&\prod_{k=1}^N\E\{\exp(i\sum_{j=1}^{N_1+N_2}b_j\eps \int_{I_k}V(s,\eps B_s^j)ds)\}\\
=&\exp(-\frac12\sum_{m,n=1}^{N_1+N_2}b_mb_n\sum_{k=1}^N\eps^2\int_{I_k^2}R(s-u,\eps B_s^m-\eps B_u^n)dsdu+\sum_{k=1}^N o(\eps^{2-\gamma_1})).
\end{aligned}
\end{equation}
Since $\sum_{k=1}^N o(\eps^{2-\gamma_1})\to 0$ as $\eps \to 0$, we have
\begin{equation}
\lim_{\eps\to 0}\left(\E\E_B\{e^{iX_\eps(t)}\}-\E_B\{e^{i\sum_{j=1}^{N_1+N_2}a_j\cdot \eps B_{t/\eps^2}^j}e^{-\frac12\sum_{m,n=1}^{N_1+N_2}b_mb_n\sum_{k=1}^N\eps^2\int_{I_k^2}R(s-u,\eps B_s^m-\eps B_u^n)dsdu}\}\right)=0.
\label{eq:simChr1}
\end{equation}

We claim that in \eqref{eq:simChr1}, $\sum_{k=1}^N\eps^2\int_{I_k^2}R(s-u,\eps B_s^m-\eps B_u^n)dsdu$ can be replaced by $\eps^2\int_{[0,t/\eps^2]^2}R(s-u,\eps B_s^m-\eps B_u^n)dsdu$. If this is true, we have
\begin{equation*}
\lim_{\eps\to 0}\left(\E\E_B\{e^{iX_\eps(t)}\}-\E_B\{e^{i\sum_{j=1}^{N_1+N_2}a_j\cdot \eps B_{t/\eps^2}^j}e^{-\frac12\sum_{m,n=1}^{N_1+N_2}b_mb_n\eps^2\int_{[0,t/\eps^2]^2}R(s-u,\eps B_s^m-\eps B_u^n)dsdu}\}\right)=0.
\end{equation*}
By the scaling property of Brownian motion and a change of variables, we have
\begin{equation}
\begin{aligned}
&\E_B\{e^{i\sum_{j=1}^{N_1+N_2}a_j\cdot \eps B_{t/\eps^2}^j}e^{-\frac12\sum_{m,n=1}^{N_1+N_2}b_mb_n\eps^2\int_{[0,t/\eps^2]^2}R(s-u,\eps B_s^m-\eps B_u^n)dsdu}\}\\
=&\E_B\{e^{i\sum_{j=1}^{N_1+N_2}a_j\cdot  B_{t}^j}e^{-\frac12\sum_{m,n=1}^{N_1+N_2}b_mb_n\eps^{-2}\int_{[0,t]^2}R(\frac{s-u}{\eps^2}, B_s^m-B_u^n)dsdu}\}.
\end{aligned}
\end{equation}
For almost every realization of $B_s^j$, we decompose $\eps^{-2}\int_0^t\int_0^tR(\frac{s-u}{\eps^2}, B_s^m-B_u^n)dsdu=(i)+(ii)$ with
\begin{eqnarray}
(i)&=&\int_0^tds \int_0^{s/\eps^2}R(u,B_s^m-B_{s-u\eps^2}^n)du,\\
(ii)&=&\int_0^tdu\int_0^{u/\eps^2}R(s, B_u^n-B_{u-s\eps^2}^m)ds.
\end{eqnarray}
By the dominated convergence theorem, we have $$(i)\to \int_0^tds\int_0^\infty R(u,B_s^m-B_s^n)du=\frac12\int_0^t\RR(B_s^m-B_s^n)ds.$$ The same limit holds for $(ii)$. So
\begin{equation}
\frac{1}{\eps^{2}}\int_0^t\int_0^tR(\frac{s-u}{\eps^2}, B_s^m-B_u^n)dsdu\to \int_0^t \RR(B_s^m-B_s^n)ds
\end{equation}
for almost every realization of $B_s^j$, which implies
\begin{equation}
\begin{aligned}
&\E_B\{e^{i\sum_{j=1}^{N_1+N_2}a_j\cdot  B_{t}^j}e^{-\frac12\sum_{m,n=1}^{N_1+N_2}b_mb_n\eps^{-2}\int_{[0,t]^2}R(\frac{s-u}{\eps^2}, B_s^m-B_u^n)dsdu}\}\\
\to &\E_B\{e^{i\sum_{j=1}^{N_1+N_2}a_j\cdot  B_{t}^j}e^{-\frac12\sum_{m,n=1}^{N_1+N_2}b_mb_n \int_0^t \RR(B_s^m-B_s^n)ds}\}
\end{aligned}
\end{equation}
and completes the proof.

What remains is to show that
\begin{equation}
|\sum_{k=1}^N\eps^2\int_{I_k^2}R(s-u,\eps B_s^m-\eps B_u^n)dsdu-\eps^2\int_{[0,t/\eps^2]^2}R(s-u,\eps B_s^m-\eps B_u^n)dsdu|\to 0.
\end{equation}
By direct calculation, we have
\begin{equation}
\begin{aligned}
LHS\les &\sum_{m\neq n}\eps^2\int_{I_m\times I_n}\sup_{x\in\R^d}|R(s-u,x)|dsdu
+\sum_{m,n=1}^N\eps^2\int_{J_m\times J_n}\sup_{x\in\R^d}|R(s-u,x)|dsdu\\
+&\sum_{m,n=1}^N\eps^2\int_{I_m\times J_n}\sup_{x\in\R^d}|R(s-u,x)|dsdu+\eps^2\int_{[0,t/\eps^2]\times [N\Delta t,t/\eps^2]}\sup_{x\in\R^d}|R(s-u,x)|dsdu\\
:=&(I)+(II)+(III)+(IV).
\end{aligned}
\end{equation}
We first assume $\varphi^\frac12(r)\les r^{-\lambda}$. For $(I)$, when $m\neq n$ and $s\in I_m,u\in I_n$, we have $|s-u|\geq \eps^{-\gamma_2}$ hence $\sup_{x\in\R^d}|R(s-u,x)|\les \varphi^\frac12(\eps^{-\gamma_2})$, so $(I)\les N^2\eps^2 \eps^{-2\gamma_1}\varphi^\frac12(\eps^{-\gamma_2})\les \eps^{\lambda \gamma_2-2}$. If we choose $\lambda$ sufficiently large (e.g., $\lambda>2/\gamma_2$), $(I)\to 0$ as $\eps\to 0$. The discussion for $(II)$ is contained in the proof of Proposition \ref{prop:conVarg2}. For $(III)$, if $|m-n|\leq 1$, we have a bound of order $N\eps^2\eps^{-\gamma_2}\sim\eps^{\gamma_1-\gamma_2}\to 0$; if $|m-n|\geq 2$, by similar discussion, we have a bound of order $\eps^2N^2\eps^{-\gamma_1-\gamma_2}\varphi^\frac12(\eps^{-\gamma_1})\les \eps^{\lambda\gamma_1-2}\to 0$. For $(IV)$, we have $(IV)\les \eps^{2-\gamma_1}\to 0$ as $\eps\to 0$. The proof is complete.
\end{proof}

Let $u_\eps(t,x)=u_{\eps,1}+iu_{\eps,2}$ and $u_0(t,x)=u_{0,1}+iu_{0,2}$, then Proposition \ref{prop:conMoment} implies that for any $N_1,N_2\in \mathbb{N}$, we  have $\E\{u_{\eps,1}^{N_1}u_{\eps,2}^{N_2}\}\to \E\{u_{0,1}^{N_1}u_{0,2}^{N_2}\}$. Since both $u_{\eps,i}$ and $u_{0,i}$ are bounded for $i=1,2$, we obtain for any $\theta_1,\theta_2$ that
\begin{equation}
\begin{aligned}
\E\{e^{i\theta_1u_{\eps,1}+i\theta_2 u_{\eps,2}}\}&=\sum_{k=0}^\infty \frac{1}{k!}\E\{(i\theta_1u_{\eps,1}+i\theta_2 u_{\eps,2})^k\}\\
&\to \sum_{k=0}^\infty \frac{1}{k!}\E\{(i\theta_1u_{0,1}+i\theta_2 u_{0,2})^k\}=\E\{e^{i\theta_1u_{0,1}+i\theta_2 u_{0,2}}\},
\end{aligned}
\end{equation}
which completes the proof of the case $\alpha=\infty$.

\section*{Acknowledgment}  

We would like to thank the anonymous referee for his careful reading of the manuscript and helpful suggestions and comments. This paper was partially funded by AFOSR Grant NSSEFF-FA9550-10-1-0194 and NSF grant DMS-1108608.

\appendix
\section{Technical Lemmas}

\begin{lemma}
Assume $|X|\leq C$ is $\F^{k+n}$ measurable, then
\begin{equation}
|\E\{X|\F_k\}-\E\{X\}|\leq 2C\varphi(n).
\label{eq:condiMixing}
\end{equation}
\label{lem:condiMixing}
\end{lemma}
\begin{proof}
First, we show \eqref{eq:condiMixing} holds with $\F_k$ replaced by any $A\in \F_k$ with $\Pb(A)>0$, i.e.,
\begin{equation}
|\E\{X|A\}-\E\{X\}|\leq 2C\varphi(n).
\end{equation}
This can be done by approximation. Let $X=\sum_k c_k1_{A_k}$ with $|c_k|$ uniformly bounded by $C$ and $A_i\cap A_j=\emptyset$ when $ i\neq j$, then $|\E\{X|A\}-\E\{X\}|\leq C\sum_k|\Pb(A_k|A)-\Pb(A_k)|$. Let $A_+$ be the union of $A_k$ such that $\Pb(A_k|A)>\Pb(A_k)$ and $A_-$ the union of the rest. Then we have
\begin{equation}
|\E\{X|A\}-\E\{X\}|\leq C\left(\Pb(A_+|A)-\Pb(A_+)+\Pb(A_-)-\Pb(A_-|A)\right)\leq 2C\varphi(n).
\end{equation}

Next we assume $\|\E\{X|\F_k\}-\E\{X\}\|_\infty>2C\varphi(n)+\eps$ for some $\eps>0$, so there exists a set $A\in \F_k$ such that $\Pb(A)>0$ and $|(\E\{X|\F_k\}-\E\{X\})1_A(\omega)|>(2C\varphi(n)+\eps)1_A(\omega)$. Without loss of generality, assume $(\E\{X|\F_k\}-\E\{X\})1_A(\omega)>(2C\varphi(n)+\eps)1_A(\omega)$, possibly for some other $A$. Integrating on both sides, we obtain $\E\{X1_A\}-\E\{X\}\Pb(A)\geq (2C\varphi(n)+\eps)\Pb(A)$, so
\begin{equation}
\E\{X|A\}-\E\{X\}\geq 2C\varphi(n)+\eps,
\end{equation}
which is a contradiction. The proof is complete.
\end{proof}

\begin{lemma}
For independent Brownian motions $B_t,W_t$ and any $\gamma>0,\beta\in (0,1)$, we have
\begin{equation}
\eps^{\gamma}\int_0^{\eps^{-\gamma}}\int_0^{\eps^{-\gamma}}R(s-u,\eps^\beta(B_s-B_u))dsdu\to 2\int_0^\infty R(t,0)dt
\end{equation}
in $L^2$ and
\begin{equation}
\eps\int_0^{t/\eps^2}\int_0^{t/\eps^2}|R(s-u,\eps^\beta(B_s-W_u))|dsdu\to 0
\end{equation}
in probability.
\label{lem:conVarSigma}
\end{lemma}
\begin{proof}
Let $\tilde{R}(t,\xi)=\int_{\R^d}R(t,x)e^{-i\xi \cdot x}dx$. For $\eps^{\gamma}\int_0^{\eps^{-\gamma}}\int_0^{\eps^{-\gamma}}R(s-u,\eps^\beta(B_s-B_u))dsdu$, by direct calculation, we have
\begin{equation}
\begin{aligned}
&\E_B\{\eps^{\gamma}\int_0^{\eps^{-\gamma}}\int_0^{\eps^{-\gamma}}R(s-u,\eps^\beta(B_s-B_u))dsdu\}\\
=& \eps^{\gamma}\int_0^{\eps^{-\gamma}}\int_0^{\eps^{-\gamma}}\int_{\R^d}\frac{1}{(2\pi)^d}\tilde{R}(s-u,\xi) e^{-\frac12|\xi|^2 \eps^{2\beta} |s-u|}d\xi dsdu\\
=&2\eps^{\gamma}\int_0^{\eps^{-\gamma}}ds\int_0^s\int_{\R^d}\frac{1}{(2\pi)^d}\tilde{R}(u,\xi)e^{-\frac12|\xi|^2\eps^{2\beta}u}d\xi du\\
=&2\int_0^1ds\int_0^{s/\eps^{\gamma}}\int_{\R^d}\frac{1}{(2\pi)^d}\tilde{R}(u,\xi)e^{-\frac12|\xi|^2\eps^{2\beta}u}d\xi du\\
\to & 2\int_0^1 ds\int_0^\infty \int_{\R^d}\frac{1}{(2\pi)^d}\tilde{R}(u,\xi)d\xi du=2\int_0^\infty R(u,0)du,
\end{aligned}
\end{equation}
and
\begin{equation}
\begin{aligned}
&\E_B\{ \left(\eps^{\gamma}\int_0^{\eps^{-\gamma}}\int_0^{\eps^{-\gamma}}R(s-u,\eps^\beta(B_s-B_u))dsdu\right)^2\}\\
=& 4\eps^{2\gamma}\int_{[0,\eps^{-\gamma}]^4}1_{s_1>u_1}1_{s_2>u_2}\E_B \{R(s_1-u_1,\eps^\beta(B_{s_1}-B_{u_1}))R(s_2-u_2,\eps^\beta(B_{s_2}-B_{u_2}))\}dsdu\\
=&4\eps^{2\gamma}\int_{[0,\eps^{-\gamma}]^4}1_{s_1>u_1}1_{s_2>u_2}\E_B \{R(u_1,\eps^\beta(B_{s_1}-B_{s_1-u_1}))R(u_2,\eps^\beta(B_{s_2}-B_{s_2-u_2}))\}dsdu\\
=&4\int_0^1ds_1\int_0^{s_1\eps^{-\gamma}}du_1\int_0^1ds_2\int_0^{s_2\eps^{-\gamma}}du_2\\
&\frac{1}{(2\pi)^{2d}}\int_{\R^{2d}} \tilde{R}(u_1,\xi_1)\tilde{R}(u_2,\xi_2) \E_B\{e^{i\xi_1\cdot \eps^\beta(B_{s_1\eps^{-\gamma}}-B_{s_1\eps^{-\gamma}-u_1})}e^{i\xi_2\cdot \eps^\beta(B_{s_2\eps^{-\gamma}}-B_{s_2\eps^{-\gamma}-u_2})}\}d\xi_1d\xi_2.
\end{aligned}
\end{equation}
For fixed $s_i,u_i$, $\xi_1\cdot\eps^\beta(B_{s_1\eps^{-\gamma}}-B_{s_1\eps^{-\gamma}-u_1})+\xi_2\cdot\eps^\beta(B_{s_2\eps^{-\gamma}}-B_{s_2\eps^{-\gamma}-u_2})\to 0$ in probability, so we have $\E_B\{ \left(\eps^{\gamma}\int_0^{\eps^{-\gamma}}\int_0^{\eps^{-\gamma}}R(s-u,\eps^\beta(B_s-B_u))dsdu\right)^2\}\to \left(2\int_0^\infty R(u,0)du\right)^2$.

Now we consider $\eps\int_0^{t/\eps^2}\int_0^{t/\eps^2}|R(s-u,\eps^\beta(B_s-W_u))|dsdu$. Without loss of generality, we can assume $R$ is positive, so
\begin{equation}
\begin{aligned}
&\E_B\{\eps\int_0^{t/\eps^2}\int_0^{t/\eps^2}|R(s-u,\eps^\beta(B_s-W_u))|dsdu\}\\
=&\eps^2\int_0^{t/\eps^2}\int_0^{t/\eps^2}\int_{\R^d}\frac{1}{(2\pi)^d}\tilde{R}(s-u,\xi)e^{-\frac12|\xi|^2\eps^{2\beta}(s+u)}d\xi dsdu\\
=&2\eps^2\int_0^{t/\eps^2}ds\int_0^sdu\int_{\R^d}\frac{1}{(2\pi)^d}\tilde{R}(u,\xi)e^{-\frac12|\xi|^2\eps^{2\beta}(2s-u)}d\xi\\
\leq &2\eps^2\int_0^{t/\eps^2}ds\int_0^sdu\int_{\R^d}\frac{1}{(2\pi)^d}\tilde{R}(u,\xi)e^{-\frac12|\xi|^2\eps^{2\beta}s}d\xi\\
=&2\int_0^{t}ds\int_0^{s\eps^{-2}}du\int_{\R^d}\frac{1}{(2\pi)^d}\tilde{R}(u,\xi)e^{-\frac12|\xi|^2\eps^{2\beta-2}s}d\xi\to 0
\end{aligned}
\end{equation}
by the dominated convergence theorem since $\beta\in (0,1)$. The proof is complete.
\end{proof}

\begin{lemma}
For fixed $\eps>0$, $y^\eps_s=\tau_{(\eps^{2-\alpha}s,B_s)}\omega$ is a stationary Markov process, ergodic with respect to the invariant measure $\Pb$.
\label{lem:enviPro}
\end{lemma}

\begin{proof}
First, since $y^\eps_s=\tau_{(\eps^{2-\alpha}(s-u),B_s-B_u)}y^\eps_u$ for any $u<s$, it is a Markov process. Next, for any $A\in \F$, we have
\begin{equation}
\begin{aligned}
\E\E_B\{1_A(y^\eps_s)\}=&\int_\Omega \int_{\R^d}1_A(\tau_{(\eps^{2-\alpha}s,x)}\omega)q_s(x)dx\Pb(d\omega)\\
=&\int_{\R^d}\int_\Omega1_A(\tau_{(\eps^{2-\alpha}s,x)}\omega)\Pb(d\omega)q_s(x)dx=\Pb(A),
\label{eq:proofErgo}
\end{aligned}
\end{equation}
so $\Pb$ is an invariant measure and $y^\eps_s$ is stationary. Now if for some $A\in \F$, we have
\begin{equation}
\E_B\{1_A(y^\eps_s)\}=\int_{\R^d}1_A(\tau_{(\eps^{2-\alpha}s,x)}\omega)q_s(x)dx
=1_A(\omega)
\end{equation}
 for all $s>0$. Since $q_s(x)>0$ for all $s>0,x\in \R^d$, we obtain that $1_A(\tau_{(\eps^{2-\alpha}s,x)}\omega)=1_A(\omega)$
for all $s>0,x\in \R^d$. Let $s\to 0$, by the strongly continuity of $T_x$, we obtain $1_A(\tau_{(0,x)}\omega)=1_A(\omega)$ for all $x\in \R^d$. Hence,
\begin{equation}
1_A(\tau_{(s,x)}\omega)=1_A(\omega)
\label{eq:proofErgo1}
\end{equation} for all $s\geq 0,x\in \R^d$. Let $\omega=\tau_{(-s,x)}\omega$, we have \eqref{eq:proofErgo1} holds for all $(s,x)\in \R^{d+1}$. Therefore, by the ergodicity of $\tau$, we conclude that $\Pb(A)=0$ or $1$ and $y_s^\eps$ is ergodic with respect to $\Pb$. The proof is complete.
\end{proof}

\begin{lemma}
If $V$ is a mean zero, stationary random field with covariance function $R(t,x)$, and $B_s$ is Brownian motion independent from $V$, then
\begin{equation}
\E\E_{B}\{\left(\eps\int_0^{t/\eps^2}V(s,B_s)ds\right)^2\}\les t\int_{\R^d}\frac{\sup_t|R(t,x)|}{|x|^{d-2}}dx.
\end{equation}
\label{lem:VarBMRS}
\end{lemma}

\begin{proof}
By direct calculation, we have
\begin{equation}
\begin{aligned}
&\E\E_{B}\{\left(\eps\int_0^{t/\eps^2}V(s, B_s)ds\right)^2\}\\
\les&\eps^2\int_0^{t/\eps^2}\int_0^s\int_{\R^d}\sup_t|R(t,x)|\frac{1}{(2\pi u)^{\frac{d}{2}}}e^{-\frac{|x|^2}{2u}}dxduds\\
\les&\eps^2\int_0^\infty du(\frac{t}{\eps^2}-u)1_{u<\frac{t}{\eps^2}}\int_{\R^d}\sup_t|R(t,x)|\frac{1}{(2\pi u)^{\frac{d}{2}}}e^{-\frac{|x|^2}{2u}}dx\\
\les&\eps^2\int_0^\infty d\lambda(\frac{t}{\eps^2}-\frac{|x|^2}{2\lambda})1_{\frac{|x|^2}{2\lambda}<\frac{t}{\eps^2}}\lambda^{\frac{d}{2}-2}e^{-\lambda}\int_{\R^d}\frac{1}{\pi^{\frac{d}{2}}}\sup_t|R(t,x)|\frac{1}{|x|^{d-2}}dx\\
\les& t\int_{\R^d}\frac{\sup_t|R(t,x)|}{|x|^{d-2}}dx.
\end{aligned}
\end{equation}
\end{proof}

\begin{lemma}
Let $X_i=(t_i,x_i)\in \R^{d+1},i=1,\ldots,4$, then under Assumption \ref{ass:mixing},
\begin{equation}
\begin{aligned}
&|\E\{V(X_1)V(X_2)V(X_3)V(X_4)\}-R(X_1-X_2)R(X_3-X_4)|\\
\leq &\Psi(X_1-X_3)\Psi(X_2-X_4)+\Psi(X_1-X_4)\Psi(X_2-X_3),
\end{aligned}
\end{equation}
where $0\leq\Psi(t,x)\leq C_\beta 1\wedge (t^2+|x|^2)^{-\beta}$ for any $\beta>0$ and some constant $C_\beta$.
\label{lem:4Moment}
\end{lemma}

\begin{proof}
See \cite[Lemma 2.3]{hairer2013random}, where the mixing property is required.
\end{proof}


\end{document}